\documentclass{amsart}
\usepackage{amsmath,amssymb,amsfonts,amsthm}
\usepackage{color}

\usepackage{graphicx}
\allowdisplaybreaks

\theoremstyle{plain}
\newtheorem{theorem}{Theorem}[section]
\newtheorem{proposition}[theorem]{Proposition}
\newtheorem{lemma}[theorem]{Lemma}

\newtheorem{question}[theorem]{Question}
\newtheorem{definition}[theorem]{Definition}
\newtheorem{claim}[theorem]{Claim}
\newtheorem{remark}[theorem]{Remark}

\theoremstyle{definition}
\theoremstyle{remark}

\begin{document}

\title
[On strong algebrability and spaceability of continuous functions]{On strong algebrability and spaceability  of continuous functions and fractal dimensions}
\author{Jia Liu${^1}$ \and Saisai Shi${^2}$}
\address{$^{1}$ Institute of Statistics and Applied Mathematics, Anhui University of Finance and Economics, 233030, Bengbu, P. R. China}
\email{liujia860319@163.com; 120140012@aufe.edu.cn}
\address{${^2}$
Institute of Statistics and Applied Mathematics, Anhui University of Finance and Economics, 233030, Bengbu, P. R. China}
\email{saisai\_shi@126.com}

\thanks{This work was supported by NSFC No. 12501109 and Scientific Research Project of Colleges and Universities in Anhui Province (2023AH050264).\\}
\date{}
\begin{center}

\end{center}
\begin{abstract}
In this paper, we investigate the strong algebrability and $(\alpha,\beta)$-lineability/spaceability of continuous functions with prescribed fractal dimensions.
For $1< s< r< t\leq2$, we define
$$H_s[0,1]=\{f\in C[0,1]:{\dim}_HG_f([0,1])=s\},$$
$$\underline{B}_r[0,1]=\{f\in C[0,1]:\underline{{\dim}}_BG_f([0,1])=r\}$$
and
$$\overline{B}_t[0,1]=\{f\in C[0,1]:\overline{{\dim}}_BG_f([0,1])=t\}.$$
We prove that $H_s[0,1]\cap\underline{B}_r[0,1]\cap\overline{B}_t[0,1]$ is both strongly $\mathfrak{c}$-algebrable and  spaceable. This complements recent findings of Bonilla et al. \cite{BFBS}, Esser et al. \cite{EMVVS}, and Liu et al. \cite{LZS}.

We prove that for any $1<s\leq t\leq2$, $H_s[0,1]\cap\overline{B}_t[0,1]$ is $(p,\mathfrak{c})$-spaceable for $p=1,2$.  We also prove that $H_s[0,1]\cap\overline{B}_t[0,1]$ is $(n,m+n)$-lineable for any $m,n\in\mathbb{N}$, thus complementing the recent work of Liu et al. \cite{LS}.

\medskip

\smallskip


\end{abstract}

\maketitle

\section{Introduction}

In the present paper, we investigate the algebraic genericity of continuous functions, with a focus on their behavior with respect to fractal dimensional properties.
Lineability, spaceability and algebrability, three key notions for characterizing algebraic genericity, were originally coined by V. I. Gurariy and they first appeared in \cite{ AGS,AS,BQ,GQ,Sea}.
Since those concepts were put forward, research in this field has quickly attracted the attention of numerous scholars. The properties of lineability, spaceability and algebrability are investigated in several contexts with interesting applications in many fields such as Real and Complex Analysis \cite{Alb,APG,AGS,GMFS,GC,EGS,Sea}, Series Theory \cite{AES,APG1,BGP,GJM,Bay}, Linear Dynamics \cite{GMSFS}, Operator Theory \cite{BDP,BF,PT}, Classical Banach Sequence Space Theory \cite{BBFP,BG,BDFP,GC,CarS,LRS}, Measure and Probability Theory \cite{BFG,CFMS,FST,MPPS}, Set Theory \cite{GS}, ODE \cite{BBFP} as well as Fractal Geometry \cite{BFBS,CFS,LS,LZS} and Multi-Fractal Analysis \cite{EJ}. With further research in recent years, a series of related stronger concepts have been successively introduced, such as strong algebrability \cite{BG}, $(\alpha,\beta)$-lineability/spaceability \cite{FPT} and pointwise lineability/spaceability \cite{PellRa}, etc.

Let us recall some notations and definitions before proceeding further.
From now on, let $\mathbb{N}$ be the set of positive integers and $\mathbb{R}$ the set of real numbers. We write $\aleph_0$ for the cardinality of $\mathbb{N}$ and $\mathfrak{c}$ for the cardinality of $\mathbb{R}$.  The space $C[0,1]$ of real-valued continuous functions from $[0,1] \to \mathbb{R}$ is equipped with the maximum norm. As is well known, $C[0,1]$ is an infinite dimensional separable Banach space and its dimension equals $\mathfrak{c}$. For a function $f: K \subseteq \mathbb{R} \to \mathbb{R}$, its graph is denoted by $G_f(K)$, i.e.,
\[
G_f(K) = \{(x, f(x)) : x \in K\} \subseteq K \times \mathbb{R}.
\]

Let $X$ be a vector space and let $\alpha$, $\beta$ be two cardinal numbers with $\alpha\leq\beta$. We say that a subset $A\subseteq X$ is
\begin{itemize}
  \item lineable if there is an infinite dimensional subspace $M$ such that $M\subseteq A\cup\{0\}$.
  \item $\alpha$-lineable if there is an $\alpha$-dimensional subspace $M$ such that $M\subseteq A\cup\{0\}$.
  \item maximal-lineable if it is ${\dim}(X)$-lineable.
    \item $(\alpha,\beta)$-lineable if it is $\alpha$-lineable and for each $\alpha$-dimensional subspace $W_\alpha \subseteq A \cup \{0\}$, there is a $\beta$-dimensional subspace $W_\beta$ such that

   \begin{equation}\label{eqdef}
    W_\alpha \subseteq W_\beta \subseteq A \cup \{0\}.
   \end{equation}
\end{itemize}

\medskip
If, in addition, $X$ is a topological vector space, then the subset $A$ is said to be

\begin{itemize}
  \item  spaceable  whenever there is a closed infinite dimensional subspace $M$
      such that $M\subseteq A\cup\{0\}$.
  \item dense-lineable  whenever there is an infinite dimensional dense  subspace $M$ such that $M\subseteq A\cup\{0\}$.
  \item $\alpha$-dense-lineable whenever there is an $\alpha$-dimensional dense subspace $M$ such that $M\subseteq A\cup\{0\}$.
  \item  maximal-dense-lineable whenever there is a ${\dim}(X)$-dimensional dense subspace $M$ such that $M\subseteq A\cup\{0\}$.
  \item $(\alpha,\beta)$-spaceable if the subspace $W_\beta$ satisfying (\ref{eqdef}) can always be chosen closed.
  \item $(\alpha,\beta)$-dense-lineable if it is $\alpha$-lineable and for each $\alpha$-dimensional subspace $W_\alpha \subseteq A \cup \{0\}$ there is a $\beta$-dimensional dense subspace $W_\beta$ such that
    \[
        W_\alpha \subseteq W_\beta \subseteq A \cup \{0\}.
    \]
\end{itemize}

And, the term algebrability is defined by a similar approach, which was introduced in \cite{AS}. If $X$ is a vector space contained in some (linear) algebra, then $A$ is called:
\begin{itemize}
  \item   algebrable if there is a subalgebra $C\subseteq X$ so that $C\subseteq A\cup\{0\}$ and $C$
has an infinite minimal system of generators. Here, by $S=\{z_{s}\}$ a minimal set of generators of $C$, we mean that $C = \mathcal{A}(S)$ is the algebra generated by $S$, and for every $s_0, z_{s_0}\notin\mathcal{A}(S\setminus \{z_{s_0} \})$.
 \item dense-algebrable if, in addition, $C$ can be taken dense in $X$.
\end{itemize}

A strengthened notion of algebrability was introduced in \cite{BG}. Given a commutative algebra $Y$, a subset $A$ of $Y$ is called
\begin{itemize}
\item strongly $\alpha$-algebrable if there exists a set $Z=\{z_{\gamma}:\gamma<\alpha\}\subseteq A$
of free generators of a subalgebra $\mathcal{B}\subseteq A\cup\{0\}$ (that is, the set $\hat{Z}$ of all elements of the form $z^{k_1}_{\alpha_1}z^{k_2}_{\alpha_2}\dots z^{k_n}_{\alpha_n}$, with nonnegative integers $k_1,\dots,k_n$ nonequal to zero simultaneously, is linearly independent and $\operatorname{span}\hat{Z}\subseteq A\cup\{0\}$).
\end{itemize}

\medskip
Recently, the notions of pointwise $\alpha$-lineability (respectively, spaceability) in \cite{PellRa} were introduced as follows:
\begin{itemize}
    \item a subset $A$ of a vector space $X$ is called pointwise $\alpha$-lineable if, for each $x \in A$, there is an $\alpha$-dimensional subspace $W_x$ such that
\begin{equation}\label{eqdef2}
        x \in W_x \subseteq A \cup \{0\}.
\end{equation}

  \item if $X$ is a topological vector space and the subspace $W_x$ satisfying (\ref{eqdef2}) can always be chosen closed (dense), we say that $A$ is respectively pointwise $\alpha$-spaceable (pointwise $\alpha$-dense-lineable). If $\alpha = \dim(X)$, we say that $A$ is respectively maximal pointwise spaceable (maximal pointwise dense-lineable).
\end{itemize}

 It is plain that the concepts of $(\alpha,\beta)$-lineability and $(\alpha,\beta)$-spaceability respectively imply
   $\beta$-lineability and $\beta$-spaceability; meanwhile, pointwise $\beta$-lineability and $\beta$-spaceability respectively imply $(1,\beta)$-lineability and $(1,\beta)$-spaceability. The reverse implications fail to hold in the general setting (see Example 2.2 in \cite{PellRa}). However, provided the set $A\cup\{0\}$ is closed under scalar products, the pointwise variants of these notions coincide exactly with their respective $(1,\beta)$ counterparts. General criteria for $(\alpha,\beta)$-lineability/spaceability can be found in \cite{AB,FPRR}.

\medskip
Over the last thirty years, numerous authors have extensively studied the pathological real-valued functions \cite{AGPS,APG,AGS,AS,GFVS,GPS,BFBS,CS1,CS2,EGS,FGK,MFSS,GFS}.
Early breakthroughs include Gurariy's proof that the set of continuous nowhere differentiable functions on $[0,1]$ is lineable \cite{Gu}.
Fonf, Gurariy and Kadets \cite{FGK} later established their spaceability in $C[0, 1]$, while
Enflo et al. \cite{EGS} proved that for every infinite dimensional closed subspace $X$ of $C[0, 1]$, the set of functions in $X$ having infinitely many zeros in $[0, 1]$ is spaceable in $X$.
The vast literature on this subject has been built during the last two decades; for a more detailed and comprehensive treatment, we refer to the survey paper \cite{GPS} or the monograph \cite{AGPS}.

\medskip
In the study of algebraic genericity for the fractal dimensions of graphs of continuous functions, Bonilla et al. \cite{BFBS} proved that for a given $s\in(1,2]$, the set of continuous functions whose graph has both Hausdorff and box dimensions $s$ is maximal-dense-lineable in $C[0,1]$ and that the set of functions whose graph has Hausdorff dimension $s$ is spaceable.
Recently, by addressing two open questions raised by Bonilla et al. \cite[Question 2.1 and Question 3.4]{BFBS}, Liu et al. \cite{LZS} proved that given $s\in(1,2]$, the set of functions $f\in C[0,1]$ whose graph has both box and Hausdorff dimensions $s$ everywhere in $[0,1]$
is maximal-dense-lineable and dense-algebrable. They raised the following question \cite[Question 5]{LZS}.
\begin{question}
 Let ${\dim}$ be one of ${\dim}_H$, $\overline{{\dim}}_B$ and $\underline{{\dim}}_B$. Let $\alpha, \beta$ be two cardinal numbers with $1<\alpha<\beta$. Given $s\in(1,2]$, is it possible to obtian the pointwise spaceability or $(\alpha,\beta)$-lineability/spaceability of the set of functions $f\in C[0,1]$ with ${\dim}G_f([0,1])=s$?
\end{question}
 Very recently, Liu et al. \cite{LS} partially answered the above question. They proved that given $s\in(1,2]$, the set of continuous functions whose graph has Hausdorff dimension $s$ is $(p,\mathfrak{c})$-spaceable for $p=1,2$ and $(n,n+m)$-lineable for any $m,n\in\mathbb{N}$. They also proved that the set of continuous functions whose graph has upper box dimension $s$ is $(\alpha,\mathfrak{c})$-spaceable if and only if $\alpha<\aleph_0$. Besides, they briefly discussed the $(\alpha,\beta)$-lineability of functions whose graph has both Hausdorff and box dimension $s$. Notably, C. Esser et al. \cite{EMVVS} proved that the set of functions $f\in C[0,1]$ with ${\dim}_HG_f([0,1])=s\in(1,2]$ is strongly $\mathfrak{c}$-algebrable.

\medskip

Inspired by the aforementioned results, we shall continue the research on the relevant issues in the present paper.
The paper is organized as follows. In forthcoming Section 2, we shall collect the definitions of Hausdorff, lower box and upper box dimensions and a series of auxiliary results which will be used throughout the paper.
In Section 3, we present our first main result (Theorem \ref{thmn1} and Theorem \ref{thmn2}) which states that for any $1< s< r< t\leq 2$, the intersection
$H_s[0,1] \cap \underline{B}_r[0,1] \cap \overline{B}_t[0,1]$
is strongly $c$-algebrable and spaceable. This complements recent findings of Bonilla et al. \cite{BFBS}, Esser et al. \cite{EMVVS} and Liu et al. \cite{LZS}.
In Section 4, we prove that for any $1<s\leq t\leq2$, $H_s[0,1] \cap \overline{B}_t[0,1]$ is $(p,\mathfrak{c})$-spaceable for $p=1,2$ and $(n,n+m)$-lineable for any $n,m\in\mathbb{N}$, thus complementing the recent work of Liu et al. \cite{LS}.
\medskip
\section{Preliminaries}

\smallskip For the convenience of the reader, we recall the definitions of Hausdorff, lower box and upper box dimensions. For a non-empty set $A \subseteq \mathbb{R}^d$, the s-dimensional Hausdorff measure is defined as
 \[
\mathcal{H}^s(A) = \lim_{\delta \to 0} \mathcal{H}^s_\delta(A),  \text{where}
\]
\[
\mathcal{H}^s_\delta(A) = \inf\left\{ \sum_{i=1}^\infty |U_i|^s : A \subseteq \bigcup_{i=1}^\infty U_i,\ \forall i \ |U_i| \leq \delta\ \right\},
\]
hereafter, for a set $U$, we let $|U|$ denote its diameter.
The Hausdorff dimension of $A$ is defined as
\[
\dim_H A = \inf\{ s \geq 0 : \mathcal{H}^s(A) = 0 \} = \sup\{ s \geq 0 : \mathcal{H}^s(A) = \infty \}.
\]

\medskip
Furthermore, given a non-empty and bounded subset $A$ of $\mathbb{R}^d$, let $N_\delta(A)$ denote the smallest number of sets of diameter at most $\delta$ needed to cover $A$. Then the lower and upper box dimensions of $A$ are defined respectively as
\[
\underline{\dim}_B A = \liminf_{\delta \to 0} \frac{\log N_\delta(A)}{-\log \delta} \quad \text{and} \quad \overline{\dim}_B A = \limsup_{\delta \to 0} \frac{\log N_\delta(A)}{-\log \delta}.
\]
If $\underline{\dim}_B A=\overline{\dim}_B A$, then we call the common value the box dimension of $A$ and denote it by $\dim_B A$, that is
\[
\dim_B A = \lim_{\delta \to 0} \frac{\log N_\delta(A)}{-\log \delta}.
\]
For more details on the above definitions and their properties, see \cite{Fal}. Notice that the relationship between box dimension and Hausdorff dimension is shown in the following inequality
\begin{equation}\label{sec2equ}
\dim_H A \leq \underline{\dim}_B A \leq \overline{\dim}_B A
\end{equation}
and the fact
\begin{equation}\label{sec2equ1}
\dim(A\times [0,1])= {\dim}A+1
\end{equation}
for any non-empty bounded subset $A$ of $\mathbb{R}^d$, where ${\dim}$ denotes any one of ${\dim}_H$, $\underline{\dim}_B$ and $\overline{\dim}_B$.

\medskip
For a bounded subset $K\subseteq\mathbb{R}$, we denote by $C(K)$ the set of all real-valued uniformly continuous functions on $K$. Assume that $s \in (0,+\infty)$ and $K \subseteq [0,1]$, denote $H_s(K)$, $H_{<s}(K)$, $\overline{B}_s(K)$,$\overline{B}_{<s}(K)$, $\underline{B}_s(K)$ and $\underline{B}_{<s}(K)$ by
\noindent
\[
H_s(K) = \left\{ f \in C(K) : {\dim}_H G_f(K) = s \right\},
\]
\[
H_{<s}(K) = \left\{ f \in C(K) : {\dim}_H G_f(K) < s \right\},
\]
\[
\overline{B}_s(K) = \left\{ f \in C(K) : \overline{\dim}_B G_f(K) = s \right\},
\]
\[
\overline{B}_{<s}(K) = \left\{ f \in C(K) : \overline{\dim}_B G_f(K) < s \right\},
\]
\[
\underline{B}_s(K) = \left\{ f \in C(K) : \underline{\dim}_B G_f(K) = s \right\},
\]
\[
\underline{B}_{<s}(K) = \left\{ f \in C(K) : \underline{\dim}_B G_f(K) < s \right\}.
\]

\medskip
\begin{definition}
Let $K \subseteq [0,1]$ and $f \in C(K)$.
If for all but countably many points $x \in K$, there exists some $\delta_x > 0$ such that $f$ is Lipschitz on $K \cap B(x,\delta_x)$, then $f$ is said to be \emph{nearly locally Lipschitz on $K$}. Here and in what follows, let $B(x,\delta)$ denote the open ball centered at $x$ with radius $\delta$.
\end{definition}

\medskip
Due to the countable stability and bi-Lipschitz invariance of the Hausdorff dimension, we can easily obtain the following lemma.
\begin{lemma} \cite{LS}\label{lemL}
Let $K \subseteq \mathbb{R}$ , $f, g \in C(K)$. If $f$ is nearly locally Lipschitz on $K$, then we have
\[
\dim_H G_{f+g}(K) = \dim_H G_g(K).
\]
In particular, the above equality holds for a Lipschitz function $f$.
\end{lemma}
By the bi-Lipschitz invariance of lower box and upper box dimensions, we can easily obtain the following lemma.
\begin{lemma}\label{lembi}
Let $K\subseteq \mathbb{R}$ be a bounded set and let $f, g \in C(K)$. If $g$ is a Lipschitz function, then
\[
\dim G_{f+g}(K) = \dim G_f(K),
\]
where $\dim$ denotes any one of $\underline{\dim}_B$ and $\overline{\dim}_B$.
\end{lemma}

\medskip
\begin{lemma}\label{lemLW}\cite{LW}
Assume that $K$ is a compact subset of $[0,1]$, $f \in C(K)$. Let $\widetilde{f}$ be the linear extension of $f$ to the whole interval $[0,1]$, i.e.,
\[
\widetilde{f}(x)=
\begin{cases}
f(x), & x \in K, \\
\text{affine}, & \text{on each component of } [0,1]\setminus K.
\end{cases}
\]
Then
\[
\dim_H G_{\widetilde{f}}([0,1]) = \max\left\{ \dim_H G_f(K), 1 \right\}.
\]
\end{lemma}
\medskip
\begin{lemma}\cite{FF}\label{lemfal-fr}
Let $f, g \in C([0,1])$. Then
\[
\overline{\dim}_B G_{f+g}([0,1])\leq \max\left\{ \overline{\dim}_B G_f([0,1]), \overline{\dim}_B G_g([0,1]) \right\}.
\]
\end{lemma}


\medskip
\begin{lemma}\label{lemLL}\cite{LL}
Let $K \subseteq [0,1]$ be a compact set. Then for any $f \in C(K)$, we have
\[
\overline{\dim}_B G_{f}(K)\leq\overline{\dim}_B G_{\widetilde{f}}([0,1]) \leq \overline{\dim}_B K + 1\]
and
\[
\underline{\dim}_B G_{f}(K)\leq\underline{\dim}_B G_{\widetilde{f}}[0,1] \leq \underline{\dim}_B K + 1,
\]
where $\widetilde{f} \in C([0,1])$ is the linear extension of $f$ to the whole interval $[0,1]$.
\end{lemma}
\begin{lemma}\cite[Theorem, Remark $1$ and $2$]{NW}\label{Ni-W}
For each triple $(s,r,t)$ of real numbers in the interval $[0,1]$ with $s <r < t$, there exists a compact subset $K \subseteq [0,1]$ such that
\[
({\dim}_H K\cap I, \underline{\dim}_B K\cap I, \overline{\dim}_B K\cap I) = (s,r,t)
\]
and $\mathcal{H}^s(K\cap I)>0$
for any non-empty open set $I$ with $K\cap I\ne\emptyset$.
\end{lemma}

\begin{lemma}\cite{Dough}\label{lemDou}
Assume that $G_1$, $G_2$ are abelian Polish groups and
$\phi : G_1 \to G_2$ is a continuous onto homomorphism.
If $S \subseteq G_2$ is prevalent then so is $\phi^{-1}(S) \subseteq G_1$.
\end{lemma}

By applying Lemma \ref{lemDou} and Tietze's extension theorem, we obtain the following lemma \cite[Corollary 2.7]{Bal}.

\begin{lemma}\cite{Bal}\label{balk3}
Assume that $K_1 \subseteq K_2$ are compact metric spaces.
Let
\[
R : C(K_2) \to C(K_1), \quad R(f) = f|_{K_1}.
\]
If $A \subseteq C(K_1)$ is prevalent then $R^{-1}(A) \subseteq C(K_2)$ is prevalent, too.
\end{lemma}
\begin{lemma}\cite{BDE,BH}\label{balk1}
Let $K\subseteq\mathbb{R}^d$ be an uncountable compact set. Then for a prevalent $f\in C(K)$, we have
\[
{\dim}_HG_f(K)={\dim}_HK+1.
\]
\end{lemma}
\begin{lemma}\cite{Bal}\label{balk2}
Let $K\subseteq\mathbb{R}^d$ be an uncountable compact set with at most finitely many isolated points. Then for a prevalent $f\in C(K)$, we have
\[
\overline{{\dim}}_BG_f(K)=\overline{{\dim}}_BK+1
\]
and
\[
\underline{{\dim}}_BG_f(K)=\underline{{\dim}}_BK+1.
\]
\end{lemma}

\begin{definition} \cite{BBF}
We say that a function $g:\mathbb{R}\to\mathbb{R}$ is \emph{exponential-like (of range $m$)} whenever $g$ is given by
\[
g(x)=\sum_{i=1}^m a_i e^{\beta_i x}
\]
for some distinct non-zero real numbers $\beta_1,\dots,\beta_m$ and some non-zero real numbers $a_1,\dots,a_m$.
\end{definition}

\begin{lemma}\label{lemexp1}\cite{BBF}
For every $m\in\mathbb{N}$, every exponential-like function $g$ of range $m$ and every $c\in\mathbb{R}$, the level set $g^{-1}(\{c\})$ has at most $m$ elements and there exists a finite decomposition of $\mathbb{R}$ into intervals such that $g$ is strictly monotone in each of them.
\end{lemma}

\begin{lemma}\label{lemexp2}\cite{BBF}
Let $\mathcal{F}$ be a family of functions from $[0,1]$ to $\mathbb{R}$ and assume that there exists an $F\in\mathcal{F}$ such that $g\circ F\in\mathcal{F}\setminus\{0\}$ for any exponential-like function $g$. Then $\mathcal{F}$ is strongly $\mathfrak{c}$-algebrable. In particular, if $H$ is a Hamel basis of $\mathbb{R}$, then
\[
\{\exp(rF):r\in H\}
\]
is a system of generators of a free algebra contained in $\mathcal{F}\cup\{0\}$.
\end{lemma}

\begin{definition} A sequence $(e_n)_{n\geq 1}$ in a Banach space $X$ is called a (Schauder) basic sequence if it constitutes a Schauder basis for $\overline{\operatorname{span}}(e_n)_{n\geq 1}$.
\end{definition}
\begin{lemma}\cite{LS}\label{lembasic}
Let $X$ be a Banach space, and let $(e_n)_{n \geq 1}$ be a basic sequence. Suppose that $\{x_i\}^m_{i=1} \subseteq X$ and  $\{\overline{x}_i = x_i + X_0: i=1,\dots,m\}$ is linearly independent in the quotient space $X / X_0$. Then the sequence $S = \{x_1,\dots, x_m, e_1, e_2, \dots\}$ is a basic sequence in $X$.
\end{lemma}


\section{Strong algebrability and spaceability of functions}

In this section, we prove that $H_s[0,1] \cap \underline{B}_r[0,1] \cap \overline{B}_t[0,1]$
is strongly $\mathfrak{c}$-algebrable (Theorem \ref{thmn1}) and spaceable (Theorem \ref{thmn2}) for any $1< s< r< t\leq2$. For this purpose, we need the following proposition.
\begin{proposition}\label{prop1}
Let $s$, $r$, $t$ be real numbers with $1 < s < r < t \leq 2$. Then there exists a compact set $K \subseteq [0,1]$ and $\widetilde{g} \in H_s[0,1] \cap\underline{B}_r[0,1] \cap \overline{B}_t[0,1]$ such that
\begin{equation}\label{equprop1}
\bigl( \dim_H K \cap I, \, \underline{\dim}_B K \cap I, \, \overline{\dim}_B K \cap I \bigr) = (s-1, r-1, t-1)
\end{equation}
and
\[
\bigl( \dim_H G_{\widetilde{g}}(I), \, \underline{\dim}_B G_{\widetilde{g}}(I), \, \overline{\dim}_B G_{\widetilde{g}}(I) \bigr) = (s, r, t)
\]
for any open interval $I \subseteq [0,1]$ with $I \cap K \neq \emptyset$.
\end{proposition}

\begin{proof}
By applying Lemma \ref{Ni-W}, we can take a compact set $K \subseteq [0,1]$ such that (\ref{equprop1}) holds for any open intervals $I \subseteq [0,1]$ with $K  \cap I\neq \emptyset$.
Clearly, $K$ contains no isolated points.
Assume that $\{U_i\}_{i=1}^\infty$ is the family of all open interval with rational endpoints that intersect $K$. Let $C_i = \overline{U_i}\cap K$,  where $\overline{U_i}$ represents the closure of the interval $U_i$. Note that
\[
\bigl( \dim_H C_i, \, \underline{\dim}_B C_i, \, \overline{\dim}_B C_i \bigr) = (s-1, r-1, t-1).
\]
It follows from Lemmas \ref{balk1} and \ref{balk2} that
\[
\mathcal{A}_i := \Bigl\{ f \in C(C_i) : \bigl( \dim_H G_f(C_i), \underline{\dim}_B G_f(C_i), \overline{\dim}_B G_f(C_i) \bigr) = (s, r, t) \Bigr\}
\]
is prevalent in $C(C_i)$. By Lemma \ref{balk3}, $R_i^{-1}(\mathcal{A}_i)$ is prevalent in $C(K)$, where
\[
R_i : C(K) \to C(C_i), \quad R_i(f) = f|_{C_i}.
\]
Furthermore, we have
\[
\bigcap_{i=1}^\infty R_i^{-1}(\mathcal{A}_i) =: \mathcal{A}
\]
is prevalent, i.e.,
\[
\mathcal{A}= \Bigl\{ f \in C(K) : \bigl( \dim_H G_f(C_i), \underline{\dim}_B G_f(C_i), \overline{\dim}_B G_f(C_i) \bigr) = (s, r, t),\ \forall i \in \mathbb{N} \Bigr\}
\]
is prevalent in $C(K)$. Since for any $f \in \mathcal{A}$ and any open interval $I \subseteq [0,1]$ with $I \cap K \neq \emptyset$, there exists $i \in \mathbb{N}$ such that $C_i \subseteq\overline{I} \cap K$,
then we obtain that
\[
\dim_H G_f(I \cap K) \geq s, \quad \underline{\dim}_B G_f(I \cap K) \geq r \quad \text{and} \quad \overline{\dim}_B G_f(I \cap K) \geq t.
\]
It follows from the monotonicity of these three dimensions that
\[
\dim G_f(I \cap K) \leq \dim G_f(K) \leq \dim K + 1,
\]
where $\dim$ denotes any one of $\dim_H$, $\underline{\dim}_B$ and $\overline{\dim}_B$.
This concludes that
\[
\bigl( \dim_H G_f(I \cap K), \underline{\dim}_B G_f(I \cap K), \overline{\dim}_B G_f(I \cap K) \bigr) = (s, r, t)
\]
for any $f \in \mathcal{A}$ and any open interval $I \subseteq [0,1]$ with $I \cap K \neq \emptyset$.

Take a function $g \in \mathcal{A}$. Let $\widetilde{g}$ be the linear extension of $g$ to the whole interval $[0,1]$, i.e.,
\[
\widetilde{g}(x) =
\begin{cases}
g(x), & x \in K, \\[4pt]
\text{affine}, & \text{on each component of } [0,1] \setminus K.
\end{cases}
\]
Then by Lemmas \ref{lemLW} and \ref{lemLL}, $\widetilde{g}$ is the desired function.
\end{proof}

\begin{theorem}\label{thmn1}
Let $s, r, t$ be real numbers with $1 < s< r <t \leq 2$. Then
\[
H_s[0,1] \cap \underline{B}_r[0,1] \cap \overline{B}_t[0,1]
\]
is strongly $\mathfrak{c}$-algebrable.
\end{theorem}

\begin{proof}
By Proposition \ref{prop1}, we obtain a compact set $K \subseteq [0,1]$ and $f \in H_s[0,1] \cap \underline{B}_r[0,1] \cap \overline{B}_t[0,1]$ such that
\[
\bigl({\dim}_H K \cap I, \underline{\dim}_B K \cap I, \overline{\dim}_B K \cap I \bigr) = (s-1, r-1, t-1)
\]
and
\begin{equation}\label{thm1equ1}
\bigl( \dim_H G_f(I), \underline{\dim}_B G_f(I), \overline{\dim}_B G_f(I) \bigr) = (s, r, t)
\end{equation}
for any open interval $I \subseteq [0,1]$ with $K\cap I \neq \emptyset$.

Let $g$ be an exponential-like function of range $m \in \mathbb{N}$. By Lemma \ref{lemexp2}, we only need to prove that
$g \circ f \in H_s[0,1] \cap \underline{B}_r[0,1] \cap \overline{B}_t[0,1]$, i.e.,
\[
\bigl( \dim_H G_{g \circ f}([0,1]), \underline{\dim}_B G_{g \circ f}([0,1]), \overline{\dim}_B G_{g \circ f}([0,1]) \bigr) = (s, r, t).
\]
Consider the map
\[
T : G_f([0,1]) \to G_{g \circ f}([0,1])\]
\[
 (x, f(x)) \mapsto (x, g(f(x))).
\]
It follows from the fact that
\[
|g(f(x)) - g(f(y))| \leq \max_{u\in f([0,1])}|g'(u)| |f(x) - f(y)|,
\]
that $T$ is a Lipschitz map. Therefore, we obtain that
\[
\dim_H G_{g \circ f}([0,1]) \leq s,
\]
\[
\underline{\dim}_B G_{g \circ f}([0,1]) \leq r
\]
and
\[
\overline{\dim}_B G_{g \circ f}([0,1]) \leq t.
\]
By Lemma \ref{lemexp1}, $(g')^{-1}(\{0\})$ is a finite set. Thus we can take an $x \in K$ satisfying $f(x) \notin (g')^{-1}(\{0\})$.
Then there exists an open interval $I \subseteq [0,1]$ such that
\[x
\in I \quad \text{and} \quad f(I) \cap (g')^{-1}(\{0\})=\emptyset.
\]
Thus $g$ is strictly monotone on $f(I)$, in particular, $g$ is invertible on $f(I)$.
Combining this with the fact that $g$ is continuously differentiable on $f(I)$ and
$f(I)\cap (g')^{-1}(\{0\})=\emptyset$, the inverse of $g$ is continuously differentiable on $f(I)$. We conclude that
\[
T : G_f(I) \to G_{g \circ f}(I)
\]
is bi-Lipschitz. Based on the bi-Lipschitz invariance of these three dimensions, we obtain that
\[
\begin{aligned}
& \bigl( \dim_H G_{g \circ f}(I), \underline{\dim}_B G_{g \circ f}(I), \overline{\dim}_B G_{g \circ f}(I) \bigr) \\
&= \bigl( \dim_H G_f(I), \underline{\dim}_B G_f(I), \overline{\dim}_B G_f(I) \bigr).
\end{aligned}
\]
Combining this with (\ref{thm1equ1}) and the upper bound of these dimensions derived above,
we conclude that $g \circ f \in H_s[0,1] \cap \underline{B}_r[0,1] \cap \overline{B}_t[0,1]$.
The proof is thus complete.
\end{proof}

\begin{theorem}\label{thmn2}
Let $1< s< r< t\leq2$. Then $H_s[0,1] \cap \underline{B}_r[0,1] \cap \overline{B}_t[0,1]$ is spaceable.
\end{theorem}

\begin{proof}
Applying Lemma \ref{Ni-W}, there exists a compact subset $K$ of $[0,1]$ such that
\[
{\dim}_H K\cap I = s-1, \quad \underline{\dim}_B K\cap I = r-1, \quad \overline{\dim}_B K\cap I = t-1
\]
for any nonempty open interval $I$ with $K\cap I\neq \emptyset$. This implies that $K$ has no isolated points.
We then choose a sequence of pairwise disjoint open intervals $(I_n)_{n\ge 1}$ such that
\begin{itemize}
    \item $I_n \cap K \neq \emptyset$ for all $n\ge 1$,
    \item $I_{n+1}$ always lies to the left of the interval $I_n$ for each $n$.
\end{itemize}
Denote by $K_n = K \cap \overline{I}_n$.
Then $K_n$ is a compact set with at most finitely many isolated points and
\[
(\dim_H K_n, \underline{\dim}_B K_n, \overline{\dim}_B K_n) = (s-1, r-1, t-1).
\]
For each $n\ge 1$, by Lemmas \ref{lemL}, \ref{lembi}, \ref{balk1} and \ref{balk2}, there exists a function $f_n \in C(K_n)$ with the following properties:
\begin{itemize}
    \item $\big(\dim_H G_{f_n}(K_n), \underline{\dim}_B G_{f_n}(K_n), \overline{\dim}_B G_{f_n}(K_n)\big) = (s, r, t)$;
    \item $f_n\big(r(K_n)\big) = f_n\big(l(K_n)\big) = 0$,
\end{itemize}
 where $r(K_n) = \max\{x \mid x\in K_n\}$ and $l(K_n) = \min\{x \mid x\in K_n\}$.
We extend $f_n$ linearly to the whole interval $[0,1]$, and denote the extention by $\widetilde{f}_n$, specifically,
\[
\widetilde{f}_n(x)=
\begin{cases}
f_n(x), & x\in K_n,\\[4pt]
\text{affine}, & \text{on each component of } \operatorname{conv}(K_n)\setminus K_n,\\[4pt]
0, & x\in [0,1]\setminus \operatorname{conv}(K_n),
\end{cases}
\]
where $\operatorname{conv}(K_n)$ denotes the convex hull of $K_n$.
One can directly observe that $\{\widetilde{f}_n\}_{n\ge1}$ have pairwise disjoint supports. So $\{\widetilde{f}_n\}_{n\ge1}$ is linearly independent.

We let
\[
M = \overline{\operatorname{span}}(\widetilde{f}_n)_{n\ge1},
\]
i.e., $M$ is the closed linear subspace of $C[0,1]$ spanned by $(\widetilde{f}_n)_{n\ge1}$. Due to the disjointness of their supports, $(\widetilde{f}_n)_{n\ge1}$ is a Schauder basic sequence. Then for any $f \in \overline{\operatorname{span}}(\widetilde{f}_n)\setminus\{0\}$, there exists a nonzero sequence $(a_n)_{n\ge1}$ of scalars such that
\[
f = \sum_{n\ge1} a_n \widetilde{f}_n.
\]

Now we proceed to prove that $f \in H_s[0,1] \cap \underline{B}_r[0,1] \cap \overline{B}_t[0,1]$. For the remainder of this proof, let $\dim$ denote any one of $\dim_H$, $\underline{\dim}_B$ and $\overline{\dim}_B$.
Since $f\neq 0$, there is some $n_0$ such that $a_{n_0} \neq 0$. Recalling that the restriction of $f$ to $K_n$ satisfies
\[f|_{K_n} = a_n \widetilde{f}_n|_{K_n} = a_n f_n\]
 for any $n\ge1$. Then we obtain that
\begin{equation}
\begin{aligned}\label{thm2eq1}
\dim G_f([0,1]) &\ge \dim G_f(K_{n_0}) = \dim G_{a_{n_0} f_{n_0}}(K_{n_0}) \\
&= \dim G_{f_{n_0}}(K_{n_0}) = \dim K_{n_0} + 1.
\end{aligned}
\end{equation}
From the construction of $\{\widetilde{f}_n\}$, it is immediate that $f$ is the linear extension of $f|_K$ to $[0,1]$. By applying Lemma \ref{lemLW} and Lemma \ref{lemLL}, we derive the corresponding upper bound
\[
\dim G_f([0,1]) \le \dim K + 1.
\]
Combining these upper bounds with (\ref{thm2eq1}), we conclude that $f \in H_s[0,1] \cap \underline{B}_r[0,1] \cap \overline{B}_t[0,1]$, which completes the proof.
\end{proof}
\section{$(\alpha,\beta)$-lineability/spaceability of functions}

In this section, we investigate the $(\alpha,\beta)$-lineability/spaceability of the intersection $H_s[0,1]\cap \overline{B}_t[0,1].$ Precisely, we prove that for any $1<s\leq t\leq2$, $H_s[0,1] \cap \overline{B}_t[0,1]$ is $(p,\mathfrak{c})$-spaceable for $p=1,2$ (Theorem \ref{thm4}) and $(n,n+m)$-lineable for any $n,m\in\mathbb{N}$ (Theorem \ref{thm3}), thus complementing two recent results of Liu et al. \cite{LS}. We first recall some relevant definitions and technical lemmas.
\begin{definition}
Suppose that $K \subseteq [0,1]$ is a compact set, $f \in C(K)$ and $x \in K$.
If
\[
\dim_H G_f(B(x,\delta) \cap K) = \dim_H G_f(K)
\]
 for any $\delta>0$,
then $x$ is said to be a \emph{full-dimensional point of $f$ for Hausdorff dimension}.
In short, we call $x$ a full-dimensional point of $f$.
\end{definition}
\begin{remark}
By the stability of Hausdorff dimension, the full-dimensional points of any function $f\in C(K)$ always exist.
\end{remark}
\begin{definition} Let $K\subseteq[0,1]$ be a compact set with positive Hausdorff dimension, and set ${\dim}_HK=s-1$. For a function $f\in C(K)$ and a point $x\in K$, we say that $x$ is a \emph{full s-dimensional point of $f$ for Hausdorff dimension with respect to $K$} if
\[
{\dim}_HG_f(K\cap B(x,\delta))={\dim}_HK+1=s
\]
holds for every $\delta>0$.
For brevity, we simply refer to such a point $x$ as \emph{a full s-dimensional point of $f$}.
\end{definition}
\begin{remark}\label{remark2} Suppose that $K\subseteq[0,1]$ is a compact set with ${\dim}_HK=s-1\in(0,1]$.

1) For a given function $f\in C(K)$,  its full $s$-dimensional points may fail to exist.

2) Let $f\in C(K)$ and $x\in K$. If $x$ is a full $s$-dimensional point then $x$ is a full-dimensional point of $f$; the converse is not true.
 If in addition, $f\in H_s(K)$ then  $x$ is a full $s$-dimensional point if and only if  $x$ is a full-dimensional point.
\end{remark}

\begin{lemma}[Restriction Theorem]\label{thm1} \cite{LS}
Let $K \subseteq [0,1]$ be a compact set and $s \in (1,2]$.
Then for any $f \in H_s(K)$, there exists a compact subset $F$ of $K$ such that
\begin{enumerate}
    \item $\dim_H F = s-1$,
    \item $f$ is nearly locally Lipschitz on $F$.
\end{enumerate}
\end{lemma}

\begin{remark}\label{remark1}
By an argument analogous to  that in \cite{LS}, if $x_0$ is a
full $s$-dimensional point of $f$, then for any $\delta>0$, there exists a
compact set $K_\delta \subseteq B(x_0,\delta) \cap K$ such that $f$ is nearly locally Lipschitz on $K_\delta$ and
\[
\dim_H K_\delta = s-1.
\]
\end{remark}

\begin{lemma}\label{cor1}\cite{LS}
Let $s \in (1,2]$ and let $K \subseteq[0,1]$ be a compact set.
If $f \in H_s(K)$ has infinitely many full-dimensional points, then there exists a sequence $\{K_j\}_{j \geq 1}$ of compact subsets of $K$  satisfying the following properties:
\begin{enumerate}
    \item $\dim_H K_j = s-1, \ \forall j \geq 1$;
    \item $\mathrm{conv}(K_i) \cap \mathrm{conv}(K_j)=\emptyset$ whenever $i \neq j$;
    \item $f$ is nearly locally Lipschitz on each $K_j$.
\end{enumerate}
Furthermore,
\[
f\Big|_{K \setminus \bigcup_{j \geq 1} K_j} \in H_s\Big(K \setminus \bigcup_{j \geq 1} K_j\Big).
\]
\end{lemma}
\begin{lemma}\cite{LS}\label{lemx}
Let $s \in (1,2]$ and let $K \subseteq[0,1]$ be a compact set with $\mathcal{H}^{s-1}(K)> 0$. Suppose $f\in H_{s}(K)$.
Then there exists a sequence  $\{K_j\}_{j \geq 1}$ of compact subsets of $K$ satisfying the following properties:
\begin{enumerate}
    \item $\dim_H K_j = s-1, \ \forall j \geq 1$;
    \item $\mathrm{conv}(K_i) \cap \mathrm{conv}(K_j) = \emptyset$ whenever $i \neq j$;
    \item $f\Big|_{K \setminus \bigcup_{j \geq 1} K_j} \in H_s\Big(K \setminus \bigcup_{j \geq 1} K_j\Big)$.
\end{enumerate}
\end{lemma}

\begin{lemma}\label{sec3prop1}\cite{LS}
Let $s\in(1,2]$, and let $\{f_i\}_{i=1}^n\subseteq \overline{B}_s[0,1]$ be linearly independent. If
$\operatorname{span}\{f_i\}_{i=1}^n \subseteq \overline{B}_s[0,1]\cup\{0\},$ then there exists a nonempty open interval $I\subseteq[0,1]$ such that
\[
\left.\sum_{i=1}^n a_i f_i\right|_{[0,1]\setminus I} \in \overline{B}_s([0,1]\setminus I)
\]
for any $(a_1,a_2,\dots,a_n)\in\mathbb{R}^n\setminus\{(0,\dots,0)\}$.
\end{lemma}

The following two propositions, which are essential to our proof in this section, are simple generalizations of \cite[Proposition 3.9 and Proposition 3.10]{LS}. Although their proofs are nearly identical to those in \cite{LS}, differing only in that full dimensional points and $P_{\{f_i\}}$ are replaced with full $s$-dimensional points and $\mathcal{P}_{\{f_i\}}$, we present them in full for the convenience of the reader.
\begin{proposition}\label{sec4prop1}
Let $s \in (1,2]$ and let $K \subseteq [0,1]$ be a compact set with ${\dim}_HK=s-1$ and $\mathcal{H}^{s-1}(K)>0$.
Suppose $f, g \in C(K)$, and define
\[\Omega_{\{f,g\}}(K)=\{(a,b)\in\mathbb{R}^2:{\dim}_HG_{af+bg}(K)=s\}.
\]
Then there exists a sequence $\{K_j\}_{j\ge1}$ of compact subsets of $K$ fulfilling the following properties:
\begin{enumerate}
    \item $\dim_H K_j = s-1$  for all $j\ge1$;
    \item $\operatorname{conv}(K_i) \cap \operatorname{conv}(K_j) = \emptyset$ whenever $i\ne j$;
    \item for any $(a,b) \in \Omega_{\{f,g\}}(K)$,
    \[
    (af+bg) \Big|_{K \setminus \bigcup_{j\ge1} K_j} \in H_s\Big(K \setminus \bigcup_{j\ge1} K_j\Big).
    \]
\end{enumerate}
\end{proposition}
\begin{proof}
We will prove it by considering the following two cases according to the number of full $s$-dimensional points of $f$ and $g$.

\medskip
\noindent \textbf{Case 1.} At least one of $f$ and $g$ has infinitely many full $s$-dimensional points.

Without loss of generality, let $f$ be the function with infinitely many full $s$-dimensional points.
By virtue of Remark \ref{remark2} and Lemma \ref{cor1}, we can take a sequence of compact subsets $\{K_j^{(1)}\}_{j\ge1}$ that satisfies the properties $(1)$ and $(2)$ of this proposition. Additionally, $f$ is nearly locally Lipschitz on each $K_j^{(1)}$. We may assume that $\mathcal{H}^{s-1}\big(K \setminus \bigcup_{j\ge1} K_j^{(1)}\big) > 0$; if not, we may instead work with an appropriate compact subset of $K_j^{(1)}$ (for some $j$) with Hausdorff dimension $s-1$ and positive $\mathcal{H}^{s-1}$-measure.

\medskip
\noindent\textbf{a1)} Suppose $\dim_H G_g\big(\bigcup_{j\ge1} K_j^{(1)}\big) < s$.

Obviously, by Lemma \ref{lemL}, we have
\[
\begin{aligned}
\dim_H G_{af+bg}\Big(\bigcup_{j\ge1} K_j^{(1)}\Big) &=\sup_{j\geq1}\dim_H G_{bg} (K^{(1)}_j) \\
&\leq\dim_H G_g\Big(\bigcup_{j\ge1} K_j^{(1)}\Big)< s
\end{aligned}
\]
for any $(a,b) \in \mathbb{R}^2$.
Hence,
\[
\dim_H G_{af+bg}\Big(K \setminus \bigcup_{j\ge1} K_j^{(1)}\Big) = s
\]
for any $(a,b) \in \Omega_{\{f,g\}}(K)$. That is, $\{K_j^{(1)}\}_{j\ge1}$ satisfies all the properties of the proposition. In this subcase, just take $K_j = K_j^{(1)}$ for each $j\geq1$.

\medskip
\noindent \textbf{a2)} Suppose $\dim_H G_g\big(\bigcup_{j\ge1} K_j^{(1)}\big) = s$.

Then for each $(a,b) \in \mathbb{R}^2$ with $b \ne 0$, we have
\begin{equation}\label{propeq1}
\dim_H G_{af+bg}\Big(\bigcup_{j\ge1} K_j^{(1)}\Big)
= \dim_H G_g\Big(\bigcup_{j\ge1} K_j^{(1)}\Big) = s.
\end{equation}
Since $\mathcal{H}^{s}\big(K \setminus \bigcup_{j\ge1} K_j^{(1)}\big) > 0$, there exists a sequence $\{K_i^{(2)}\}_{i\ge1}$ of compact sets such that
\[
K_i^{(2)} \subseteq K \setminus \bigcup_{j\ge1} K_j^{(1)},\ \ \dim_H K_i^{(2)} = s-1
\]
for all $i\ge1$ and
\[
\operatorname{conv}(K_i^{(2)}) \cap \operatorname{conv}(K_j^{(2)}) = \emptyset \ \text{whenever} \ i\ne j,
\]
\[
\mathcal{H}^{s-1}\Big(K \setminus \Big(\bigcup_{j\ge1} K_j^{(1)} \cup \bigcup_{i\ge1} K_i^{(2)}\Big)\Big) > 0.
\]

If $f\in H_{<s}(K)$, then clearly $\Omega_{\{f,g\}}(K)=\{(a,b):b\ne 0\}$. Combining these with (\ref{propeq1}), it suffices to take $K_j=K^{(2)}_j$ for each $\geq 1$.

If $f\in H_s(K)$ (this implies that $(a,0)\in\Omega_{\{f,g\}}(K)$ for any $a\ne 0$, combining this with (\ref{propeq1}), we get $\Omega_{\{f,g\}}=\mathbb{R}^2\setminus\{(0,0)\}$ ) and $\dim_H G_f\big(\bigcup_{i\ge1} K_i^{(2)}\big) < s$,  since $f$ is nearly locally Lipschitz on each $K_j^{(1)}$,
\[
\dim_H G_f\Big(\bigcup_{j\ge1} K_j^{(1)}\Big)
= \sup_{j\ge1} \dim_H G_f(K_j^{(1)}) = s-1 < s.
\]
Therefore,
\[
\dim_H G_f\Big(K \setminus \Big(\bigcup_{j\ge1} K_j^{(1)} \cup \bigcup_{i\ge1} K_i^{(2)}\Big)\Big) = s.
\]
Hence
\[
af\big|_{K \setminus \bigcup_{j\ge1} K_j^{(2)}}\in H_s(K \setminus \bigcup_{j\ge1} K_j^{(2)})
\]
for any $a\ne 0$.
Combining these with (\ref{propeq1}), we get that
\[
\dim_H G_{af+bg}\Big(K \setminus \bigcup_{j\ge1} K_j^{(2)}\Big) = s
\]
for any $(a,b) \in \Omega_{\{f,g\}}(K)$.
At this moment, we take $K_j = K_j^{(2)}$ for each $j$. Then $\{K_j\}_{j\geq 1}$ satisfies all the properties of this proposition.

 If $f\in H_s(K)$  and $\dim_H G_f\big(\bigcup_{i\ge1} K_i^{(2)}\big) = s$, then take any sequence $\{K_j\}_{j\ge1}$ of compact sets such that
 \[K_j \subseteq K \setminus \Big(\bigcup_{j\ge1} K_j^{(1)} \cup \bigcup_{i\ge1} K_i^{(2)}\Big),\]
\[\dim_H K_j = s-1 \]
for all $j\geq1$ and
\[\operatorname{conv}(K_i) \cap \operatorname{conv}(K_j)=\emptyset, \ \forall i\neq j.\]
It is straightforward that $\{K_j\}_{j\geq 1}$ also satisfies the property $(3)$ of this proposition.

\medskip
\noindent \textbf{Case 2.} Both $f$ and $g$ have finitely many full $s$-dimensional points.

\medskip
Define $\mathcal{P}_{\{f,g\}}$ by
\[
\mathcal{P}_{\{f,g\}} = \left\{ x \in K : \exists (a,b)\in\Omega_{\{f,g\}}(K)\ \text{s.t. } x \text{ is a full $s$-dimensional point of } af+bg \right\}.
\]

Note that the set $\mathcal{P}_{\{f,g\}}$ may possibly be empty. If this situation occurs, it suffices to pick a sequence $\{K_j\}_{j\geq 1}$ of compact subsets of $K$ such that the properties $(1),(2)$ of this proposition hold, and then property $(3)$ holds automatically. Hence we may assume that $\mathcal{P}_{\{f,g\}}\neq\emptyset$.

\medskip
\noindent\textbf{b1)} Suppose $\# \mathcal{P}_{\{f,g\}} < +\infty$.

Since $\# \mathcal{P}_{\{f,g\}}$ is finite, without loss of generality, we may write it as $\{0<x_1 < x_2 < \dots < x_n<1\}$.
Consider the open intervals $I_k=(x_{k-1},x_k)$, $k=1,2,\dots,n+1$, where $x_0=0$, $x_{n+1}=1$.
By the assumption that $\mathcal{H}^{s-1}(K) > 0$, there exists some index $ k_0$ such that
\[
\mathcal{H}^{s-1}(K \cap I_{k_0}) > 0.
\]
Choose a compact subset $F$ of $K \cap I_{k_0}$ with $0 < \mathcal{H}^{s-1}(F) < \infty$. We claim that
\[
 \dim_H G_{af+bg}(F) < s
\]
for any $(a,b) \in \Omega_{\{f,g\}}(K)$. Indeed, if this failed, then $\mathcal{P}_{\{f,g\}} \cap F \ne \emptyset$ which yields a contradiction.
Hence we can take a sequence $\{K_j\}_{j\geq 1}$ of compact subsets of $F$ satisfying the following conditions:
\[
\dim_H K_j = s-1, \ \forall j\ge1
\]
and
\[
\operatorname{conv}(K_i) \cap \operatorname{conv}(K_j) = \emptyset, \ \forall i\ne j.
\]
Then $\{K_j\}_{j\ge1}$ is exactly the required sequence.

\medskip
\noindent\textbf{ b2)} Suppose $\# \mathcal{P}_{\{f,g\}} = +\infty$.

Take a sequence of pairwise distinct points $\{x_n\}_{n\ge1} \subseteq\mathcal{P}_{\{f,g\}}$.
For each $n\geq 1$, choose $(a_n,b_n) \in \Omega_{\{f,g\}}(K)$ such that $x_n$ is a full $s$-dimensional point of $a_n f + b_n g$.
Passing to a subsequence if necessary, we may assume that $\{(a_n,b_n)\}_{n\ge1}$ is pairwise linearly independent. Otherwise, there exists $(a',b') \in \Omega_{\{f,g\}}(K)$ such that $a'f+b'g$ has infinitely many full $s$-dimensional points, at this point, we consider the function pairs $\{a'f+b'g, g\}$ or $\{a'f+b'g, f\}$, which reduces to Case 1.
We may also assume that $a_n \ne 0$, $\forall n\ge1$ (or $b_n \ne 0$, $\forall n\ge1$) by taking a subsequence.

For each $n \in \mathbb{N}$, choose $r_n>0$ such that the family $\{B(x_n,r_n)\}_{n\ge1}$ is pairwise disjoint.
By Remark \ref{remark1}, there exists a compact set $F_n \subseteq B(x_n,r_n)\cap K$ satisfying
\begin{equation}\label{propeq2}
\dim_H F_n = s-1, \quad \mathcal{H}^{s-1}\Big(K \setminus \bigcup_{n\ge1} F_n\Big) > 0,
\end{equation}
and such that $a_n f + b_n g$ is nearly locally Lipschitz on $F_n$.
Since $a_n \ne 0$ for any $n\in\mathbb{N}$, we decompose
\[
af + bg = \frac{a}{a_n}(a_n f + b_n g) + \Big(b - \frac{b_n}{a_n}a\Big)g.
\]
Consequently, we obtain
\begin{equation}\label{propeq3}
\dim_H G_{af+bg}(F_n) =
\begin{cases}
\dim_H G_g(F_n), & a b_n-b a_n  \ne 0, \\
s-1, & a b_n-b a_n = 0
\end{cases}
\end{equation}
for any $(a,b)\in\mathbb{R}^2$.

If $\dim_H G_g\big(\bigcup_{n\ge1} F_n\big) < s$, then
\[
\dim_H G_{af+bg}\Big(\bigcup_{n\ge1} F_n\Big)
= \sup_{n\ge1} G_{af+bg}(F_n) < s.
\]
This further yields
\[
\dim_H G_{af+bg}\Big(K \setminus \bigcup_{n\ge1} F_n\Big) = s
\]
for any $(a,b)\in \Omega_{\{f,g\}}(K)$.
Combining this with (\ref{propeq2}) and the disjointness of the balls $\{B(x_n,r_n)\}_{n\geq 1}$, it suffices to set $K_j=F_j$ for all $j\in\mathbb{N}$.

If $\dim_H G_g\big(\bigcup_{n\ge1}F_n\big)=s$,
we may assume without loss of generality that $\dim_H G_g(F_n) < s$ for any $n\in\mathbb{N}$. Otherwise, suppose $\dim_H G_g(F_n)=s$ for some integer $n$. By Lemma \ref{thm1}, we consider the restriction $g|_{F_n}$, which yields a compact subset $F_n'\subseteq F_n$  with $\dim_H F'_n=s-1$ such that $g$ is nearly locally Lipschitz on $F'_n$. Consequently, $\dim_H G_g(F'_n) < s$. If $\dim_H G_g(F_n)<s$ already holds, we simply set $F_n'=F_n$ and revisit the value of $\dim_H G_g\big(\bigcup_{n\ge1}F_n'\big)$). Since $\{(a_n,b_n)\}_{n\ge1}$ is pairwise linearly independent, for any given $(a,b)\in \mathbb{R}^2\setminus\{(0,0)\}$, the inequality
\[
 ab_n -ba_n \ne 0
\]
holds for all but at most one index $n$ (denoted by $n_0$ if such an index exists).
Therefore, in view of (\ref{propeq3}), we have for any $(a,b)\in \mathbb{R}^2\setminus\{(0,0)\}$,
\[
\begin{aligned}
\dim_H G_{af+bg}\Big(\bigcup_{n\ge1} F_n\Big)
&= \sup_n \dim_H G_{af+bg}(F_n) \\
&= \max\Big\{s-1, \sup_{n\ne n_0} \dim_H G_g(F_n)\Big\} \quad (\text{if } n_0 \text{ exists}) \\
&=\max\Big\{s-1,  \dim_H G_g(\bigcup_{n\neq n_0}F_n)\Big\}  \\
&= \dim_H G_g\big(\bigcup_{n\ge1} F_n\big) = s.
\end{aligned}
\]
Since $\mathcal{H}^{s-1}\big(K \setminus \bigcup_{n\ge1} F_n\big) > 0$, we may choose a sequence of compact sets $\{K_j\}_{j\ge1}$ satisfying
\[
K_j \subseteq K \setminus \bigcup_{n\ge1} F_n,
\]
\[
\dim_H K_j = s-1,
\]
for any $j\geq 1$ and
\[
\operatorname{conv}(K_i) \cap \operatorname{conv}(K_j) = \emptyset, \ \forall i\ne j.
\]
It follows immediately that
\[
\dim_H G_{af+bg}\Big(K \setminus \bigcup_{j\ge1} K_j\Big)
\ge \dim_H G_{af+bg}\Big(\bigcup_{j\ge1} F_j\Big) = s
\]
holds for any $(a,b) \in \mathbb{R}^2 \setminus \{(0,0)\}$, and of course holds for all $(a,b) \in \Omega_{\{f,g\}}(K)$.
\end{proof}

\begin{proposition}\label{sec4prop2}
Let $s\in(1,2]$ and $n\in\mathbb{N}$. Let $K\subseteq[0,1]$ be a compact set with $\mathcal{H}^{s-1}(K) > 0$. Suppose that $\{f_i\}_{i=1}^n\subseteq C(K)$. Then for any $m\in\mathbb{N}$, there exist $m$ compact subsets $\{K_j\}_{j=1}^m$ of $K$ satisfying the following properties:
\begin{enumerate}
\item $\dim_H K_j=s-1$, $\forall$ $j=1,2,\dots,m$;
\item $\operatorname{conv}(K_i)\cap \operatorname{conv}(K_j)=\emptyset$ whenever $i\neq j$;
\item for any $(a_1,\dots,a_n)\in\mathbb{R}^n$,
\[
\left.\sum_{i=1}^n a_i f_i\right|_{\bigcup_{j=1}^m K_j}
\in H_{<s}\Big(\bigcup_{j=1}^m K_j\Big).
\]
\end{enumerate}
\end{proposition}
\begin{proof}
Choose a compact set $F\subseteq K$ with ${\dim}_HF=s-1$ and $\mathcal{H}^{s-1}(F) > 0$, and by passing to the restriction functions $\{f_i|_F\}_{1\leq i\leq n}$, we may further assume without loss of generality that ${\dim}_HK=s-1$. The proof is split into two cases, based on whether the functions in $\{f_i\}_{i=1}^n$ admit infinitely many full $s$-dimensional points.

\medskip
\noindent\textbf{Case 1.}
At least one function in $\{f_i\}_{i=1}^n$ admits infinitely many full $s$-dimensional points. By rebeling if necessary, we take $f_1$ to be such a function.

By Lemma \ref{cor1}, there exist compact subsets $\{K_j^{(1)}\}_{j=1}^m$ satisfying
\begin{itemize}
\item $\dim_H K_j^{(1)}=s-1,\ \forall\,j=1,2,\dots,m$;
\item $\operatorname{conv}(K_i^{(1)})\cap \operatorname{conv}(K_j^{(1)})=\emptyset$ whenever $i\neq j$;
\item $f_1$ is nearly locally Lipschitz on each $K_j^{(1)}$.
\end{itemize}

\begin{claim}\label{claim}
For each $j\in\{1,2,\dots,m\}$, there exists a compact set $K_j\subseteq K_j^{(1)}$ such that
\[
\dim_H K_j=s-1 \quad \text{and} \quad
\left.\sum_{i=1}^n a_i f_i\right|_{K_j}\in H_{<s}(K_j)
\]
for any $(a_1,\dots,a_n)\in\mathbb{R}^n$.
\end{claim}

\noindent\textbf{Proof of the Claim.}
Fix any $j\in\{1,2,\dots,m\}$.

\textbf{a1)} Suppose that for any $(a_2,\dots,a_n)\in\mathbb{R}^{n-1}$, $\displaystyle\left.\sum_{i=2}^n a_i f_i\right|_{K_j^{(1)}}
\in H_{<s}(K_j^{(1)})$.

Since $f_1$ is nearly locally Lipschitz on $K_j^{(1)}$, Lemma \ref{lemL} implies that
\[
\left.\sum_{i=1}^n a_i f_i\right|_{K_j^{(1)}}\in H_{<s}(K_j^{(1)})
\]
for any $(a_1,\dots,a_n)\in\mathbb{R}^n$.
In this case, we simply set $K_j=K_j^{(1)}$.

\textbf{a2)} Suppose that there exists $(a_2^{(1)},\dots,a_n^{(1)})\in\mathbb{R}^{n-1}$ such that
\[
\left.\sum_{i=2}^n a_i^{(1)} f_i\right|_{K_j^{(1)}}\in H_s(K_j^{(1)}).
\]
Note that we have $(a_2^{(1)},\dots,a_n^{(1)})\neq(0,\dots,0)$. Without loss of generality, we assume that $a_2^{(1)}\neq0$. By Lemma \ref{thm1}, there exists a compact subset $K_j^{(2)}\subseteq K_j^{(1)}$ with
$\dim_H K_j^{(2)}=s-1$
on which
$\sum_{i=2}^n a_i^{(1)} f_i$  is nearly locally Lipschitz.

Observe that for any $(a_2,\dots,a_n)\in\mathbb{R}^{n-1}$,
\[
\sum_{i=2}^n a_i f_i
= \frac{a_2}{a_2^{(1)}} \sum_{i=2}^n a_i^{(1)} f_i
+ \sum_{i=3}^n \left(a_i - \frac{a_2}{a_2^{(1)}}a_i^{(1)}\right) f_i.
\]
If $\displaystyle\left.\sum_{i=3}^n a_i f_i\right|_{K_j^{(2)}}
\in H_{<s}(K_j^{(2)})$ holds for all $(a_3,\dots,a_n)\in\mathbb{R}^{n-2}$, it suffices to take $K_j=K_j^{(2)}$.
Otherwise, we iterate the procedure described in a2) and a1). By induction, we find some integer $k_0\in\{1,2,\dots,n\}$, nested compact sets $K_j^{(1)}\supseteq K_j^{(2)}\supseteq\dots\supseteq K_j^{(k_0)}$  with $\dim_H K_j^{(k)}=s-1$, and non-zero coefficients $\big(a_k^{(k-1)},a_{k+1}^{(k-1)},\dots,a_n^{(k-1)}\big)$ such that for each $k$, the linear combination $\displaystyle\left.\sum_{i=k}^n a_i^{(k-1)} f_i\right|_{K_j^{(k-1)}}$ belongs to $H_s\big(K_j^{(k-1)}\big)$ and is nearly locally Lipschitz on $K^{(k)}_j$.
If $k_0<n$, we additionally have
\[
\left.\sum_{i=k_0+1}^n a_i f_i\right|_{K_j^{(k_0)}}
\in H_{<s}\big(K_j^{(k_0)}\big),\quad \forall (a_{k_0+1},\dots,a_n)\in\mathbb{R}^{n-k_0}.
\]
Then $K_j^{(k_0)}$ is the desired compact set. Indeed, for any $(a_1,\dots,a_n)\in\mathbb{R}^n$, we can write $\sum_{i=1}^n a_i f_i$ as a combination of $f_1$, the linearly locally Lipschitz functions, and a remainder in $H_{<s}(K^{k_0}_j)$.
By Lemma \ref{lemL}, we get that
\[
\left.\sum_{i=1}^n a_i f_i\right|_{K_j^{(k_0)}}\in H_{<s}(K_j^{(k_0)}).
\]
Thus the  claim holds with $K_j=K_j^{(k_0)}$.

\medskip
Back to Case 1.
By the construction of $\{K_j\}_{j=1}^m$, the properties (1) and (2) in this proposition are obviously satisfied. Combining this with Claim \ref{claim} and the stability of Hausdorff dimension, property (3) follows as well.

\medskip
\noindent\textbf{Case 2.}
Each function $f_i$ has finitely many full $s$-dimensional points.

Define
\[
\mathcal{P}_{\{f_i\}} = \Big\{ x\in K : \exists\ \text{non-zero}\ (a_i)
\ \text{s.t.}\ x \text{ is a full $s$-dimensional point of } \sum_{i=1}^n a_i f_i \Big\}.
\]

\textbf{b1)} If $\#\mathcal{P}_{\{f_i\}}<+\infty$,
by an argument analogous to that in Proposition \ref{sec4prop1}, we can construct compact sets $\{K_j\}_{j=1}^m$ satisfying all the required properties.

\medskip
\textbf{b2)} If $\#\mathcal{P}_{\{f_i\}}=+\infty$,
take $m$ pairwise distinct points $x_1,x_2,\dots,x_m\in\mathcal{P}_{\{f_i\}}$. For each $j=1,\dots,m$, choose a non-zero vector
$(a_1^{(j)},a_2^{(j)},\dots,a_n^{(j)})$
such that $x_j$ is a full $s$-dimensional point of $\displaystyle\sum_{i=1}^n a_i^{(j)} f_i$.
Without loss of generality, assume $a_1^{(j)}\neq0$ for all $j\in\{1,2,\dots,m\}$.
Choose $r_j>0$ so that the balls $B(x_j,r_j)$ are pairwise disjoint.
By Remark \ref{remark1}, there exist compact sets $K_j'\subseteq B(x_j,r_j)\cap K$ such that
\begin{equation}\label{prop2eq1}
\dim_H K_j'=s-1, \quad \operatorname{conv}(K_i')\cap \operatorname{conv}(K_j')=\emptyset,\ \forall\,i\neq j,
\end{equation}
and the function
\[
g_j=\sum_{k=1}^n a_k^{(j)} f_k\]
is nearly locally Lipschitz on $K_j'$.
For each $j\in\{1,2,\dots,m\}$, consider the family $\{g_j,f_2,\dots,f_n\}$
(note that $\operatorname{span}\{g_j,f_2,\dots,f_n\}=\operatorname{span}\{f_1,f_2,\dots,f_n\}$),  by applying Claim \ref{claim},  we obtain a compact subset $K_j\subseteq K_j'$ such that
\begin{equation}\label{prop2eq2}
\dim_H K_j=s-1
\end{equation}
and
\[
 \left.\Big(a_1 g_j + \sum_{k=2}^n a_k f_k\Big)\right|_{K_j}\in H_{<s}(K_j),\ \forall\ (a_1,\dots,a_n)\in\mathbb{R}^n,
\]
which further implies
\[
\left.\sum_{k=1}^n a_k f_k\right|_{K_j}\in H_{<s}(K_j)
\]
for all $(a_1,\dots,a_n)\in\mathbb{R}^n$.
This concludes that
\begin{equation}\label{prop2eq3}
\left.\sum_{k=1}^n a_k f_k\right|_{\bigcup_{j=1}^m K_j}
\in H_{<s}\Big(\bigcup_{j=1}^m K_j\Big),\ \forall\ (a_1,\dots,a_n)\in\mathbb{R}^n.
\end{equation}
It follows from (\ref{prop2eq1}), (\ref{prop2eq2}), (\ref{prop2eq3}) and the hypotheses of the proposition that the family $\{K_j\}_{j=1}^m$ fulfills all the required assertions of the proposition. This completes the proof.
\end{proof}
Now we state and prove our main results of this section.
\begin{theorem}\label{thm4} Let $1<s\leq t\leq 2$. Then $H_s[0,1]\cap \overline{B}_t[0,1]$ is $(p, \mathfrak{c})$-spaceable for $p=1,2$.
\end{theorem}
\begin{proof} We only prove the case $p=2$; the case $p=1$ follows  from Lemma \ref{lemx} via an analogous argument. Let $f_1,f_2\in H_s[0,1]\cap \overline{B}_t[0,1]$ be linearly independent and
\begin{equation}\label{thm3equn2}
\operatorname{span}\{f_1,f_2\}\subseteq \big(H_s[0,1]\cap \overline{B}_t[0,1]\big)\cup\{0\}.
\end{equation}
In view of Lemma \ref{sec3prop1}, there exists a nonempty open interval $I\subseteq[0,1]$ such that
\begin{equation}\label{Sec4eq2}
(af_1+bf_2)|_{[0,1]\setminus I} \in \overline{B}_t([0,1]\setminus I)
\end{equation}
for any $(a,b)\neq(0,0)$. 
Take a compact subset $K\subseteq I$ such that
\[{\dim}_HK=s-1,\ \overline{\dim}_BK=t-1\ \text{and}\ \mathcal{H}^{s-1}(K)>0.
\]
Such a set exists by Lemma \ref{Ni-W} when $s\neq t$, if $s=t$, we may instead take a central cantor set with an appropriate contraction ratio.
Let
\[
\Omega_{\{f_1,f_2\}}(K)=\{(a,b)\in\mathbb{R}^2:af_1+bf_2\in H_s(K)\}.
\]
By Proposition \ref{sec4prop1}, there exists a sequence  $\{K^{(1)}_j\}_{j\geq 1}$ of compact subsets of $K$ such that
\begin{enumerate}
    \item $\dim_H K^{(1)}_j = s-1$, $\forall j\ge1$;
    \item $\operatorname{conv}(K^{(1)}_i) \cap \operatorname{conv}(K^{(1)}_j) = \emptyset$  whenever $i\ne j$;
    \item for any $(a,b) \in \Omega_{\{f_1,f_2\}}(K)$,
    \begin{equation}\label{Thm3equx}
    (af_1+bf_2) \Big|_{K \setminus \bigcup_{j\ge1} K^{(1)}_j} \in H_s\Big(K \setminus \bigcup_{j\ge1} K^{(1)}_j\Big).
    \end{equation}
\end{enumerate}
By removing at most countably many points from each $K^{(1)}_j$, we may assume that  each $K^{(1)}_j$ has no isolated points. Furthermore, by passing to a subsequence, we may assume that the sequence $\{K^{(1)}_j\}_{j\geq 1}$ satisfies the following properties.
\begin{itemize}
  \item The sets $\Big\{K^{(1)}_j\Big\}_{j\geq 1}$ are arranged in left-to-right (or right-to-left) order along the interval $I$.
Consequently, $\Big\{K^{(1)}_j\Big\}_{j\geq 1}$ converges to some singleton set $\{x_0\}\subseteq K$ in the Hausdorff metric and furthermore,
\[
\overline{\bigcup_{j\geq 1}K^{(1)}_j}=\Big(\bigcup_{j\geq 1}K^{(1)}_j\Big)\cup\{x_0\}.
\]
  \item For each $j\geq 1$,
  \begin{equation}\label{equation1}
  \operatorname{dist}\big(K^{(1)}_j,\{x_0\}\big)\leq 2^{-j},
  \end{equation}
  where $\operatorname{dist}$ stands for the Hausdorff metric.
\end{itemize}
It follows from \cite[Lemma 3.4]{MO} that we may choose a compact set $E\subseteq (0,1)$ with no isolated points such that
\[
{\dim}_HE=0 \quad\text{and}\quad \overline{\dim}_BE=t-1.
\]
For each $j$, we define a compact subset $E_j\subseteq \operatorname{conv}(K^{(1)}_j)$ as follows:
\[
E_j=c_jE+x_j,
\]
where $c_j=\min\{2^{-j},\frac{1}{2}\left|\operatorname{conv}\big(K^{(1)}_j\big)\right|\}$ and $x_j$ is the left endpoint of the interval $\operatorname{conv}(K_j)$. That is, $E_j$ is an affine copy of $E$ and is contained in $\operatorname{conv}(K^{(1)}_j)$. Hence
\[\dim_HE_j=0\quad  \text{and}\quad \overline{\dim}_BE_j=t-1
\]
for each $j\geq 1$.
It is plain that
$\lim_{j\to\infty}E_j=\{x_0\}$ with respect to Hausdorff metric and thus
 \[
 \overline{\bigcup_{j\geq 1}E_j}=\Big(\bigcup_{j\geq 1}E_j\Big)\cup\{x_0\}.
 \]
Clearly, ${\dim}_H\overline{\bigcup_{j\geq 1}E_j}=0$.
It follows from the definition of $\{E_j\}_{j\geq 1}$ and (\ref{equation1}) that
\[
N_{2^{-n}}\Big(\bigcup_{j\geq 1}E_j\Big)\leq N_{2^{-n}}\Big(\bigcup_{1\leq j\leq n}E_j\Big)+N_{2^{-n}}\Big(\bigcup_{ j\geq n+1}E_j\Big)\leq nN_{2^{-n}}(E_1)+1.
\]
Therefore, we obtain that
\[\overline{\dim}_B\overline{\bigcup_{j\geq 1}E_j}=\overline{\dim}_B \bigcup_{j\geq 1}E_j=t-1.\]
We then define $K_j$ by
\[
K_j=E_j\cup K^{(1)}_j, \forall j\geq 1.
\]
Set $F=\overline{\bigcup_{j\geq 1} K_j}$ (in general, we do not have $F\subseteq K$). Then it is clear that
\begin{itemize}
\item $F=\big(\bigcup_{j\geq 1} K^{(1)}_j\big)\cup\big(\bigcup_{j\geq 1}E_j\big)\cup\{x_0\}\subseteq I$;
\item ${\dim}_HF=\dim_H K_j=s-1$  and $\overline{\dim}_BF=\overline{{\dim}}_BK_j=t-1$;
\item $\operatorname{conv}(K_i)\cap \operatorname{conv}(K_j)=\emptyset$ whenever $i\neq j$;
\item for any $(a,b)\in\mathbb{R}^2\setminus\{(0,0)\}$,
\begin{equation}\label{thm3equn1}
(af_1+bf_2)\Big|_{[0,1]\setminus F} \in H_{s}\Big([0,1]\setminus F\Big).
\end{equation}
\end{itemize}
We only verify the last property stated above. Indeed, if $(a,b)\in\Omega_{\{f_1,f_2\}}(K)$, it follows from ${\dim}_H\overline{\bigcup_{j\geq 1}E_j}=0$ and (\ref{Thm3equx}) that
\[
\begin{aligned}
\dim_HG_{af_1+bf_2}([0,1]\setminus F)&\geq\dim_HG_{af_1+bf_2}(K\setminus F)\\
&=\dim_HG_{af_1+bf_2}\Big(K\setminus \Big[\Big(\bigcup_{j\geq 1}K^{(1)}_j\Big)\cup\overline{\bigcup_{j\geq 1} E_j}\Big]\Big)\\
&=\dim_HG_{af_1+bf_2}\Big(K\setminus \bigcup_{j\geq 1}K^{(1)}_j\Big)=s.
\end{aligned}
\]
 If $(a,b)\notin\Omega_{\{f_1,f_2\}}(K)\setminus\{(0,0)\}$, it follows from $af_1+bf_2\in H_s([0,1])$ that
\[ \dim_HG_{af_1+bf_2}([0,1]\setminus K)=s.
\]
Therefore,
\[
\begin{aligned}
\dim_HG_{af_1+bf_2}([0,1]\setminus F)&\geq\dim_HG_{af_1+bf_2}\Big([0,1]\setminus \big(K\cup\big(\bigcup_{j\geq 1}E_j\big)\big)\Big) \\
&=\dim_HG_{af_1+bf_2}([0,1]\setminus K)=s.
\end{aligned}
\]
Combining these with (\ref{thm3equn2}), we conclude that (\ref{thm3equn1}) holds for any $(a,b)\in\mathbb{R}^2\setminus\{(0,0)\}$.

From now on, let $I_j = \operatorname{conv}(K_j)$ for each $j \in \mathbb{N}$.
By Lemmas \ref{balk1} and \ref{balk2}, for each $j \in \mathbb{N}$, we can choose a function $g_j \in C(K_j)$ satisfying
\[
\dim_H G_{g_j}(K_j) = s\  \text{and}\ \overline{\dim}_B G_{g_j}(K_j) = t.
\]
By applying Lemma \ref{lemL}, we may further assume that
\[
g_j(r_{I_j}) = g_j(l_{I_j}) = 0,
\]
where $r_{I_j}$ and $l_{I_j}$ denote the right and left endpoints of the interval $I_j$, respectively.
We extend $g_j$ linearly to the interval $[0,1]$, and denote the corresponding extension by $\widetilde{g}_j$, i.e.,
\[
\widetilde{g}_j(x) =
\begin{cases}
g_j(x), & x \in K_j, \\
0, & x \in [0,1] \setminus I_j, \\
\text{affine}, & \text{on each component of } I_j \setminus K_j.
\end{cases}
\]
By Lemmas \ref{lemLW} and \ref{lemLL}, we obtain that
\[\dim_H G_{\widetilde{g}_j}[0,1] = \dim_H G_{g_j}(K_j) = s\]
 and
 \[
 \overline{\dim}_BG_{\widetilde{g}_j}[0,1]=\overline{\dim}_B G_{g_j}(K_j)=t.
 \]
So $\widetilde{g}_j\in H_{s}[0,1]\cap\overline{B}_t[0,1]$.
Accordingly, we obtain a sequence of continuous functions $\{\widetilde{g}_j:j \geq 1\}\subseteq  H_{s}[0,1]\cap\overline{B}_t[0,1]$.
It is clear that $\{\operatorname{supp}(\widetilde{g}_j)\}_{j \geq 1}$ is pairwise disjoint, where $\operatorname{supp}(\widetilde{g}_j)$ denotes the support of $\widetilde{g}_j$, i.e.,
\[
\operatorname{supp}(\widetilde{g}_j) = \{ x \in [0,1] : \widetilde{g}_j(x) \neq 0 \}.
\]
It is clear that $\{\widetilde{g}_j:j\geq 1\}$ is linearly independent. Furthermore, $(\widetilde{g}_j)_{j\geq 1}$ forms a Schauder basic sequence. Denote $X_0 = \overline{\operatorname{span}}(\widetilde{g}_j)_{j \geq 1}$.
It is not difficult to verify that $\overline{f}_1, \overline{f}_2$ are linearly independent in the quotient space $C[0,1]/ X_0$. Indeed, suppose for contradiction that $\overline{f}_1$, $\overline{f}_2$ are linearly independent in the quotient space; then there exists $(a,b) \neq (0,0)$ such that
\begin{equation}\label{thm3eq4}
a f_1 + b f_2:=h \in X_0.
\end{equation}
Then there exists a sequence of scalars $(a_j)_{j \geq 1}$ such that
\[
h = \sum_{j=1}^\infty a_j \widetilde{g}_j.
\]
Then by the definitions of $\widetilde{g}_j, j\geq 1$, we have
\[
\left. h \right|_{[0,1] \setminus \bigcup_{j \geq 1} I_j} = 0 \quad \text{and} \quad \left. h \right|_{I_j} = a_j \widetilde{g}_j|_{I_j}, \ j \geq 1.
\]
It follows from the fact that $\widetilde{g}_j$ is affine on each component of $I_j \setminus K_j$ that
\[
\dim_H G_h(I_j \setminus K_j)= 1.
\]
Therefore we get
\[
\begin{aligned}
&\dim_H G_h\Big([0,1] \setminus \bigcup_{j \geq 1} K_j\Big)\\
=&\dim_H G_h\Big(\Big([0,1] \setminus \bigcup_{j \geq 1} I_j \Big)\ \cup \Big( \bigcup_{j \geq 1} I_j \setminus K_j\Big)\Big) \\
=&\max\Big\{ \dim_H G_h\Big([0,1] \setminus \bigcup_{j \geq 1} I_j\Big), \ \sup_{j \geq 1} \dim_H G_h(I_j \setminus K_j) \Big\}=1.
\end{aligned}
\]
Combining this estimation with (\ref{thm3equn1}) and (\ref{thm3eq4}), we reach a contradiction.
By applying Lemma \ref{lembasic}, the sequence $\mathcal{S} = \{f_1, f_2, \widetilde{g}_1, \widetilde{g}_2, \dots\}$ forms a Schauder basic sequence.

Next, we need to verify that $\overline{\operatorname{span}} \mathcal{S} \subseteq (H_{s}[0,1]\cap\overline{B}_t[0,1]) \cup \{0\}$.
Let $f \in \overline{\operatorname{span}} \mathcal{S}  \setminus \{0\}$. Then there exists a nonzero sequence $(a_n)_{n \geq 1}$ of scalars, such that
\[
f=a_1f_1+a_2f_2+\sum_{n\geq 3}a_n\widetilde{g}_{n-2}.
\]
\noindent\textbf{First case:} $a_i \neq 0$ for some $i=1,2$.

 We rewrite $f$ by
    \[
    f= a_1f_1 + a_2f_2 + g,
    \]
    where $g = \sum_{n \geq 3} a_n \widetilde{g}_{n-2}$.
    By the constructions of $\{K_j\}$,
   $
    \dim_H\left( \bigcup_{j \geq 1} K_j \right) = s - 1.
   $
    Hence, by (\ref{sec2equ}) and (\ref{sec2equ1}), we obtain
    \begin{equation}\label{thm3eq5}
    \dim_H G_f\Big( \bigcup_{j \geq 1} K_j \Big) \leq s.
    \end{equation}
    Since $g$ is affine on each component of $I_j \setminus K_j$ and
    \[
    \left. f \right|_{[0,1] \setminus \bigcup_{j \geq 1} I_j} = \left.\Big( a_1 f_1 + a_2 f_2\Big) \right|_{[0,1] \setminus \bigcup_{j \geq 1} I_j},
    \]
 together with Lemma \ref{lemL}, we obtain that
    \[
    \begin{aligned}
    &\dim_H G_f\Big([0,1] \setminus \bigcup_{j \geq 1} K_j\Big)= \dim_H G_f\Big(\Big([0,1] \setminus \bigcup_{j \geq 1} I_j \Big) \cup \Big( \bigcup_{j \geq 1} I_j \setminus K_j\Big)\Big) \\
    &= \max\Big\{ \dim_H G_{a_1 f_1 + a_2 f_2}\Big([0,1] \setminus \bigcup_{j \geq 1} I_j\Big), \ \sup_{j \geq 1} \dim_H G_f(I_j \setminus K_j) \Big\}\\
    &= \max\Big\{ \dim_H G_{a_1 f_1 + a_2 f_2}\Big([0,1] \setminus \bigcup_{j \geq 1} K_j\Big), \ \sup_{j \geq 1} \dim_H G_{a_1 f_1 + a_2 f_2}(I_j \setminus K_j) \Big\} \\
    &= \dim_H G_{a_1 f_1 + a_2 f_2}\Big([0,1] \setminus \bigcup_{j \geq 1} K_j\Big) = s.
    \end{aligned}
    \]
Combining this with (\ref{thm3eq5}), we have $f \in H_s[0,1]$. It follows from the fact that  $g$ is a linear extension of $g|_F$ to the interval [0,1] and $\overline{\dim}_BF=t-1$, together with Lemma \ref{lemLL}, that
\[\overline{\dim}_BG_g([0,1])\leq t.
\]
Combining this with Lemma \ref{lemfal-fr} and (\ref{thm3equn2}), we obtain that
\[\overline{\dim}_BG_f([0,1])\leq t.
\]
By applying (\ref{Sec4eq2}) and the fact that $f|_{[0,1]\setminus I}=(a_1f_1+a_2f_2)|_{[0,1]\setminus I}$, we have
\[
\overline{\dim}_BG_f([0,1])\geq \overline{\dim}_BG_f([0,1]\setminus I)=\overline{\dim}_BG_{a_1f_1+a_2f_2}([0,1]\setminus I)=t.
\]

\medskip
\noindent\textbf{Second case:} $a_i = 0$ for any $i=1,2$.

 Since $f \neq 0$, there exists $i_0 \geq 3$ such that $a_{i_0} \neq 0$. It follows from
   $ \left. f \right|_{K_{i_0}} = \left. a_{i_0} \widetilde{g}_{i_0} \right|_{K_{i_0}}$
that
    \[
    \dim G_f([0,1])\geq \dim G_f(K_{i_0})= \dim G_{g_{i_0}}(K_{i_0}) ={\dim}K_{i_0}+1,
    \]
    where $\dim$ denotes any one of ${\dim}_H$ and $\overline{\dim}_B$.
    On the other hand,
    since $f$ is the linear extension of $f|_F$ to the interval $[0,1]$, by Lemmas \ref{lemLW} and \ref{lemLL},
    \[
    \overline{\dim}_B G_f([0,1])\leq\overline{\dim}_BF+1=t
    \]
    and
    \[
   {\dim}_H G_f([0,1])=\max\{{\dim}_HG_f(F),1\}\leq {\dim}_HF+1=s.
    \]
   Thus in this subcase, we have $f\in H_s[0,1]\cap\overline{B}_t[0,1]$.

The proof is then complete.
\end{proof}
\begin{theorem}\label{thm3} Let $1<s\leq t\leq 2$ and $m,n\in\mathbb{N}$. Then $H_s[0,1]\cap \overline{B}_t[0,1]$ is $(n,m+n)$-lineable.
\end{theorem}
\begin{proof}
Let $\{f_i\}^n_{i=1}\subseteq H_s[0,1]\cap\overline{B}_t[0,1]$ be linearly independent and
\begin{equation}\label{thmequx1}
\operatorname{span}\{f_i\}_{1\leq i\leq n}\subseteq (H_s[0,1]\cap\overline{B}_t[0,1])\cup\{0\}.
\end{equation}
In view of Lemma \ref{sec3prop1}, there exists a nonempty open interval $I\subseteq[0,1]$ such that
\begin{equation}\label{Sec4eq1}
\left.\sum_{i=1}^n a_i f_i\right|_{[0,1]\setminus I} \in \overline{B}_t([0,1]\setminus I)
\end{equation}
for any $(a_1,\dots,a_n)\neq(0,\dots,0)$. Take a compact subset $K\subseteq I$ such that
\[{\dim}_HK=s-1,\ \overline{\dim}_BK=t-1\ \text{and}\ \mathcal{H}^{s-1}(K)>0.
\]
From Proposition \ref{sec4prop2}, there exist $m$ compact subsets of $K$, denoted by $\{K^{(1)}_j\}_{j=1}^m$ satisfying the following properties:
\begin{itemize}
\item $\dim_H K^{(1)}_j=s-1$, $\forall$ $j=1,2,\dots,m$;
\item $\operatorname{conv}(K^{(1)}_i)\cap \operatorname{conv}(K^{(1)}_j)=\emptyset$ whenever $i\neq j$;
\item for any $(a_1,\dots,a_n)\in\mathbb{R}^n$,
\[
\left.\sum_{i=1}^n a_i f_i\right|_{\bigcup_{j=1}^m K^{(1)}_j}
\in H_{<s}\Big(\bigcup_{j=1}^m K^{(1)}_j\Big).
\]
\end{itemize}
We can remove countably many points from $K^{(1)}_j$ without changing the above three properties, so we may assume $K^{(1)}_j$ contains no isolated points.
For each $j\in\{1,2,\dots,m\}$, take a compact subset $K^{(2)}_j\subseteq\operatorname{conv}(K^{(1)}_j)$ with no isolated points and such that
\[{\dim}_HK^{(2)}_j=0\ \ \text{and}\ \overline{{\dim}}_BK^{(2)}_j=t-1.
\]
This gives that
\[
{\dim}_HG_{\sum_{i=1}^n a_i f_i}\Big(\bigcup_{j=1}^m K^{(2)}_j\Big)\leq 1
\]
for any $(a_1,\dots,a_n)\in\mathbb{R}^n$.
We define $K_j=K^{(1)}_j\bigcup K^{(2)}_j$ for each $j=\{1,2,\dots,m\}$. Then we obtain a family of compact sets $\{K_j\}_{1\leq j\leq m}$ such that for all $j=1,2,\dots,m$,
\begin{itemize}
\item $K_j\subseteq I$ contains no isolated points;
\item $\dim_H K_j=s-1$ and $\overline{{\dim}}_BK_j=t-1$;
\item $\operatorname{conv}(K_i)\cap \operatorname{conv}(K_j)=\emptyset$ whenever $i\neq j$;
\item for any $(a_1,\dots,a_n)\in\mathbb{R}^n$,
\begin{equation}\label{thmequx}
\left.\sum_{i=1}^n a_i f_i\right|_{\bigcup_{j=1}^m K_j} \in H_{<s}\Big(\bigcup_{j=1}^m K_j\Big).
\end{equation}
\end{itemize}
By (\ref{thmequx1}), (\ref{Sec4eq1}) and (\ref{thmequx}), we obtain that
\begin{equation}\label{Sec4equn2}
\left.\sum_{i=1}^n a_i f_i\right|_{[0,1]\setminus \bigcup^m_{j=1}K_j} \in H_s\Big([0,1]\setminus \bigcup^m_{j=1}K_j\Big)\bigcap\overline{B}_t\Big([0,1]\setminus \bigcup^m_{j=1}K_j\Big)
\end{equation}
for any $(a_1,\dots,a_n)\neq(0,\dots,0)$.
By applying Lemmas \ref{balk1} and \ref{balk2}, we take a function $ g_j\in C(K_j)$  such that
\[
{\dim}_HG_{g_j}(K_j)=s\ \text{and}\  \overline{{\dim}}_BG_{g_j}(K_j)=t
\]
for each $j\in\{1,2,\dots,m\}$.
We define $\tilde{g}_j\in C[0,1]$ by
\[
\widetilde{g}_j(x) =
\begin{cases}
g_j(x), & x \in K_j, \\
\text{affine}, & \text{on each component of } [0,1] \setminus K_j.
\end{cases}
\]
Then by Lemmas \ref{lemLW} and \ref{lemLL}, we have $\widetilde{g}_j\in H_s[0,1]\cap\overline{B}_t[0,1]$ for each $j\in\{1,2,\dots,m\}$.

In what follows, we need to prove that the family $\mathcal{S}=\{f_1,\dots,f_n,\widetilde{g}_1\dots,\widetilde{g}_m\}$ is linearly independent and
\[
\operatorname{span}\mathcal{S}\subseteq\big( H_s[0,1]\cap\overline{B}_t[0,1]\big)\cup\{0\}.
\]
We first prove that $\mathcal{S}$ is linearly independent.
Assume that
\begin{equation}\label{thm3equan1}
h := \sum_{i=1}^n a_i f_i + \sum_{j=1}^m b_j \widetilde{g}_j = 0.
\end{equation}
Since $\widetilde{g}_j$ is affine on each component of $[0,1]\setminus K_j$, the sum
$\sum_{j=1}^m b_j \widetilde{g}_j$
is affine on each component of $[0,1]\setminus \bigcup_{j=1}^m K_j$.
Hence by  Lemma \ref{lemL},
\[
\dim_H G_h\Big([0,1]\setminus \bigcup_{j=1}^m K_j\Big) = \dim_H G_{\sum_{i=1}^n a_i f_i}\Big([0,1]\setminus \bigcup_{j=1}^m K_j\Big).
\]
Combining this with (\ref{Sec4equn2}) and (\ref{thm3equan1}), we get that
$a_i=0, i=1,2,\dots,n.$
This implies that
\[
h = \sum_{j=1}^m b_j \widetilde{g}_j = 0.
\]
If $b_{j_0} \neq 0$ for some $j_0 \in \{1,2,\dots,m\}$, since
$\widetilde{g}_j$ is affine on each component of $[0,1]\setminus K_j$, together with the disjointness of $\big\{\operatorname{conv}(K_{j})\big\}_{1\leq j\leq n}$, we obtain that
\[
\sum_{j\neq j_0} b_j \widetilde{g}_j \text{ is affine on } \operatorname{conv}(K_{j_0}).
\]
Therefore,
\[
\begin{aligned}
\dim_H G_h[0,1] &\geq \dim_H G_h\left(K_{j_0}\right) = \dim_H G_{\widetilde{g}_{j_0}}\left(K_{j_0}\right) \\
&= \dim_H G_{g_{j_0}}(K_{j_0}) = s.
\end{aligned}
\]
This contradicts $h=0$. Hence $b_j=0, \forall j$. This concludes that the family $\mathcal{S}$ is linearly independent. Let $f\in\operatorname{span}\mathcal{S}\setminus\{0\}$.
Then
\[
f= \sum_{i=1}^n a_i f_i + \sum_{j=1}^m b_j \widetilde{g}_j
\]
for some nonzero vector $(a_1,\dots,a_n,b_1,\dots,b_m)$.

If $(a_1,\dots,a_n)\ne(0,\dots,0)$, it follows from the fact that $\sum^m_{j=1}b_j\widetilde{g}_j$ is affine on each component of $[0,1]\setminus\bigcup_{1\leq j\leq m}K_j$, together with ${\dim}_H{\bigcup^m_{j=1} K_j}=s-1$, (\ref{sec2equ1}), (\ref{Sec4eq1}), (\ref{Sec4equn2}) and Lemmas \ref{lemL}, \ref{lembi} and \ref{lemLL} that
\[
{\dim}_HG_f[0,1]=s\ \text{and}\ \overline{{\dim}}_BG_f([0,1])=t.
\]
Indeed,
\[
\begin{aligned}
\dim_H G_{f}([0,1])&=\max\Big\{\dim_H G_{f}\Big([0,1]\setminus\bigcup^m_{j=1}K_j\Big),\dim_H G_{f}\Big(\bigcup^m_{j=1}K_j\Big)\Big\}\\
&=\max\Big\{\dim_H G_{\sum^n_{i=1}a_if_i}\Big([0,1]\setminus\bigcup^m_{j=1}K_j\Big),\dim_H G_{f}\Big(\bigcup^m_{j=1}K_j\Big)\Big\} \\
&=s
\end{aligned}
\]
and
\[
\overline{{\dim}}_B G_{f}([0,1])\geq\overline{{\dim}}_B G_{f}([0,1]\setminus I)=\overline{{\dim}}_B G_{\sum^n_{i=1}a_if_i}([0,1]\setminus I)
=t.
\]
It follows from Lemma \ref{lemLL} and $\overline{\dim}_B\big(\bigcup^m_{j=1}K_j\big)=t-1$ that
\begin{equation}\label{sec4eq3}
\overline{{\dim}}_B G_{\sum^m_{j=1}b_j\widetilde{g}_j}([0,1])\leq t.
\end{equation}
Combining this with $\overline{{\dim}}_B G_{\sum^n_{i=1}a_if_i}([0,1])\leq t$ and Lemma \ref{lemfal-fr}, we obtain that  $\overline{{\dim}}_B G_f([0,1])\leq t$.

If $(a_1,\dots,a_n)=(0,\dots,0)$, then there exists $j_0$ such that $b_{j_0}\ne 0$. By the construction of $\widetilde{g}_j$, we have
\[
f|_{K_{j_0}}=b_{j_0}g_{j_0}.
\]
Furthermore,
\[
{\dim}G_f([0,1])\geq{\dim}G_f(K_{j_0})={\dim}G_{g_{j_0}}(K_{j_0})={\dim}K_{j_0}+1,
\]
where ${\dim}$ denotes any one of ${\dim}_H$ and $\overline{\dim}_B$.
By inequality (\ref{sec4eq3}), we have $\overline{\dim}_BG_f([0,1])\leq t$.
It follows from the fact that $\sum^m_{j=1}b_j\widetilde{g}_j$ is affine on each component of $[0,1]\setminus\bigcup_{j}K_j$ and Lemma \ref{lemLL} that ${\dim}_HG_f([0,1])\leq s$.

 The proof is thus complete.
 \end{proof}

\end{document}